\documentclass[12pt]{amsart}
\usepackage[mathscr]{eucal}
\usepackage{amsfonts}
\usepackage{amssymb,amsmath,amsthm}
\theoremstyle{plain}
\newtheorem{theorem}{Theorem}
\newtheorem{lemma}[theorem]{Lemma}

\theoremstyle{definition}

\begin{document}

\title{Geometric Structure of Sumsets}
\author{Jaewoo Lee}
\address{Department of Mathematics\\ Borough of Manhattan Community College\\ The City University of New York\\ 199 Chambers Street\\ New York, NY 10007}
\email{jlee@bmcc.cuny.edu}
\date{April 25, 2007}
\begin{abstract}
 Given a finite set of lattice points, we compare its sumsets and lattice points in its dilated convex hulls. Both of these are known to grow as polynomials. Generally, the former are subsets of the latter. In this paper, we will see that sumsets occupy all the central lattice points in convex hulls, giving us a kind of approximation to lattice points in polytopes.
\end{abstract}

\maketitle

\section{Introduction}\label{S:SecIntro}

Let $A$ be a a set of lattice points in dimension $n$. For any nonnegative integer $h$, we define the \emph{h-fold sumset} $hA = \{ a_1 + a_2 + \ldots + a_h : a_1, a_2, \ldots, a_h \in A\}$. And the \emph{dilation} is
$ h*A = \{ ha : a \in A\}$. 

A \emph{hyperplane} $H$ is the set $\{x \in \mathbb{R}^n : (x,u)=\alpha\}$ for some nonzero $u \in \mathbb{R}^n$ and some number $\alpha$, where (\textperiodcentered\,,\,\textperiodcentered) indicates an inner product in $\mathbb{R}^n$\,. The vector $u$ is called a \emph{normal vector} to $H$. A hyperplane divides $\mathbb{R}^n$ into two closed half-spaces $H^+$ and $H^-$ where
\[ H^+ = \{x \in \mathbb{R}^n : (x,u)\ge\alpha\} \]
\[ H^- = \{x \in \mathbb{R}^n : (x,u)\le\alpha\} \,. \]

We write $d(x,y)$ to denote the distance between two points $x, y \in \mathbb{R}^n$\,. If $S, T \subseteq \mathbb{R}^n$\,, then 
\[ d(x, S) = \inf_{s \in S} d(x,s) \,, \]
\[d(S,T)= \inf_{s\in S, t \in T} d(s,t)\,.\] In particular, the distance from a point $x \in \mathbb{R}^n$  to a hyperplane $H$ where $x \notin H$, is given by the length of the perpendicular line segment from $x$ to $H$.

If two hyperplanes $H_1, H_2$ are parallel, their normal vectors are multiples of each other, so we can take a single normal vector $u$ and write $H_1=\{ x : (x,u) = \alpha_1 \}$ and $H_2 = \{ x : (x,u) = \alpha_2 \}$. Take any $x \in H_1$\,. Then $d(x, H_2)$  is given by the perpendicular line segment. To calculate the distance between $H_1$ and $H_2$, note that $x + tu$ where $t \in \mathbb{R}$ gives the perpendicular ray from $x$ to $H_2$\,. If the ray meets $H_2$ when $t=t_2$\,, then $t_2 = (\alpha_2 - \alpha_1 )/|u|^2$. Thus, $d(x, H_2) = |t_2 u|=(\alpha_2 - \alpha_1 )/|u| $\,, which is independent of the choice of $x$. Therefore, when $H_1$ and $H_2$ are parallel, $d(H_1, H_2)$ is given by the length of any perpendicular line segment joining them.

A \emph{polytope} is the convex hull of a finite set of points in some $\mathbb{R}^n$\,, or equivalently, a bounded set which is an intersection of finitely many closed half-spaces. Let $\Delta = \text{conv}(a_1, a_2, \ldots , a_m)$ where $a_i \in \mathbb{Z}^n$. Then define the \emph{dilation of $\Delta$}\,, $h*\Delta$\,, as  
\begin{eqnarray*}
h*\Delta & = & \{ hx : x \in \Delta \} \\
 & = & \{\,\sum \lambda_i a_i : \lambda_i \ge 0,\, \sum \lambda_i =h\}\\
 & = & \text{conv}( ha_1, \ldots, ha_m )\,.
\end{eqnarray*}

Assume that $\Delta \subseteq \mathbb{R}^n$ is an $n$-dimensional nonempty lattice polytope, and $h$ is a positive integer. Then Ehrhart~\cite{Ehr} showed that there is a polynomial $p(h)$, called the \emph{Ehrhart polynomial}, such that 
 \[ |(h*\Delta ) \cap \mathbb{Z}^n | = p(h) \] where
\[ p(h) =\text{Vol}(\Delta) h^n + \frac{\text{Vol}(\partial(\Delta))}{2} h^{n-1} + \cdots + \chi(\Delta) \,. \]
Here, $\chi(\Delta)$ is the Euler characteristic of $\Delta$, and Vol($\partial(\Delta)$) is the surface area of $\Delta$ normalized with respect to the sublattice on each face of $\Delta$.

If $\Delta = \text{conv}(A)$ where $A$ is a finite set of integral points in some $\mathbb{R}^n$, then $|(h*\Delta ) \cap \mathbb{Z}^n | \ge |hA|$. Then we can consider the growth of $|hA|$ instead. Nathanson~\cite{NathGr} proved that the growth of $|hA|$ is a linear function when $A$ is a subset of integers normalized in a way. When $A_1 , A_2, \ldots , A_r$ and $B$ are finite subsets of $\mathbb{N}_0$, normalized similarly as above, then Han, Kirfel, and Nathanson~\cite{HKN} showed that $| B+ h_1 A_1 + \cdots +h_r A_r |$ is a multilinear function of $h_1 , \ldots , h_r$ eventually. Furthermore, if $A_1 , A_2, \ldots , A_r$ and $B$ are finite subsets of an abelian semigroup which contains $0$, then $| B+ h_1 A_1 + \cdots +h_r A_r |$ is a polynomial of $h_1 , \ldots , h_r $ for all sufficiently large $h_1, \ldots , h_r$\,, which was proved by Khovanski\u\i~\cite{Kho} when $r=1$, and by Nathanson~\cite{NathMul} for $r \ge 2$. And if $A, B$ are finite subsets of an abelian group without elements of finite order, then Khovanski\u\i~\cite{Kho} computed the degree and the leading coefficient of the polynomial above (actually his proof contained a gap, but the gap is fixed by the work of this paper). 

In this paper, we will investigate the growth of the sumset $hA$ from the geometric point of view.

\section{Khovanski\u\i's Lemmas}\label{S:Khowo}

Our paper starts with lemmas that Khovanski\u\i \ proved in \cite{Kho}. Let $A$ be a finite subset of $\mathbb{Z}^n$\,, $A=\{a_1, \ldots, a_m\}$\,, with $|A|=m$ and $\Delta = \text{conv}(A)$\,. Also assume that $A$ generate $\mathbb{Z}^n$ as a group. 

\begin{lemma}\label{L:Khoone}
There exists a constant $C$ with the following property: for all linear combination $\sum \lambda_i a_i$ of $a_i \in A$ with real coefficients $\lambda_i$ such that $\sum \lambda_i a_i$ is an integral point, there exists a linear combination $\sum n_i a_i$ of $a_i$ with integer coefficients $n_i$ such that $\sum n_i a_i = \sum \lambda_i a_i$\,, with $\sum |n_i - \lambda_i| < C$\,.
\end{lemma}

\begin{proof}
Let $X= \{x : x \in \mathbb{Z}^n, x=\sum \lambda_i a_i\,,\  \text{with}\ 0 \le \lambda_i \le 1 \}   $\,, which is a finite set. Since $A$ generate $\mathbb{Z}^n$\,, each $x \in X$ can be written as $x= \sum_{i=1}^m n_i(x) a_i$\,, where $n_i(x) \in \mathbb{Z}$\,. So for each $x \in X$\,, we fix one representation $\sum_{i=1}^m n_i(x) a_i$ with $n_i(x) \in \mathbb{Z}$\,. Let $q= \max_{x \in X} \sum_{i=1}^m |n_i(x)|$ and let $C=m+q$\,, a positive integer. Then for any $z=\sum \lambda_i a_i \in \mathbb{Z}^n$\,, $x= z-\sum [\lambda_i] a_i \in X$\,. So $x=\sum_{i=1}^m n_i(x) a_i$ with $n_i(x) \in \mathbb{Z}$ and $z= \sum_{i=1}^m \bigl( n_i(x) + [\lambda_i] \bigr) a_i= \sum_{i=1}^m \lambda_i a_i$ with
$\sum |n_i(x) +[\lambda_i] -\lambda_i| < \sum_{i=1}^m \bigl(|n_i(x)|+1 \bigr) \le q+m =C$\,.
\end{proof}

Let $h$ be a positive integer and assume $0 \in A$. Then \[\Delta = \{ \,\sum \lambda_i a_i : \lambda_i \ge 0,\, \sum \lambda_i \le 1 \}\] and \[h*\Delta = \{\,\sum \lambda_i a_i : \lambda_i \ge 0,\, \sum \lambda_i \le h\}\,. \]
 Define \[\Delta(h,C)=\{\,\sum \lambda_i a_i : \lambda_i \ge C\,,\, \sum \lambda_i \le h-C\}\] with $C$ as in Lemma~\ref{L:Khoone}.

Then, if $x=\sum \lambda_i a_i \in \Delta(h,C)$\,, let $\lambda_i = \alpha_i + C, \,\alpha_i \ge 0$\,. So 
\begin{eqnarray*}
\Delta(h,C) & = & \{\, \sum (\alpha_i +C) a_i : \,  \alpha_i \ge 0\,, \\
 & & \ \ \ \ \ \ \ \ \sum \alpha_i \le h-C-mC\} \\
 & = & C\sum a_i +\{\,\sum \alpha_i a_i : \, \alpha_i \ge 0\,,\\
 & & \ \ \ \ \ \ \ \ \sum \alpha_i \le h-C-mC \}\\
 & = & C\sum a_i + (h-C-mC)*\Delta\,.
\end{eqnarray*}
Note $\Delta(h,C)$ is an empty set when $h < C+mC$, a single point $C\sum a_i$ when 
$h= C+mC$, and a dilation of $\Delta$ translated by an integral point when $h \ge 
C+mC+1$\,.

Let $\mathbb{Z}^n (A)$ be the group generated by the differences of the elements of $A$\,.

\begin{lemma}\label{L:Khosec}
Assume $\mathbb{Z}^n (A) = \mathbb{Z}^n$, and\, $0 \in A$\,. Then, every integral point in $\Delta(h,C)$ belongs to the sumset $hA$\,.
\end{lemma}

\begin{proof}
Let $z$ be an integral point in $\Delta(h,C)$\,. Then \[z= \sum \lambda_i a_i, \ \ \lambda_i \ge C\,, \,\sum \lambda_i \le h-C\,.\] By Lemma~\ref{L:Khoone}, $z=\sum n_i a_i\,,\  n_i \in \mathbb{Z}\,,\, \sum | n_i - \lambda_i | < C$\,. If $n_i < 0$ for some $i$, then $| n_i - \lambda_i | > C$. Therefore, all $n_i$ must be nonnegative. And $\sum n_i = \sum | n_i | = \sum | n_i - \lambda_i + \lambda_i | \le \sum | n_i - \lambda_i |\, + \, \sum | \lambda_i | < C+h-C = h$\,. Thus $z= \sum n_i a_i\,,\  n_i \ge 0\,, \,\sum n_i < h$\,. Since $0 \in A$\,, \[hA=\{\, \sum n_i a_i : n_i \ge 0, \,\sum n_i \le h \}\,,\] therefore $ z \in hA$\,.
\end{proof}

Using these results, Khovanski\u\i \ in \cite{Kho} gave an argument for the following theorem, but the argument contained an error about boundary points. For details and how to modify his arguments to get a similar result on simplex, see \cite{Jae}.

\begin{theorem}\label{T:Dist}
Suppose $\mathbb{Z}^n (A) = \mathbb{Z}^n$. Then, there exists a constant $\rho
$ with the following property: for any positive integer $h$, every integral point of $h*\Delta$, whose distance to $\partial(h*\Delta)$ is more than 
$\rho
$\,, belongs to $hA$. 
\end{theorem}

The condition $\mathbb{Z}^n (A) = \mathbb{Z}^n$ implies that the dimension of $\Delta$ is $n$. In general, $hA$ is a proper subset of $(h*\Delta) \cap \mathbb{Z}^n$. Theorem~\ref{T:Dist} states that $hA$ takes all of the central region in $h*\Delta$.

\section{Proof of Theorem}\label{S:Proof}

Now we prove Theorem~\ref{T:Dist}.

\begin{proof}[PROOF OF THEOREM~\ref{T:Dist}]
Take any hyperplane \[H=\{ x : (x,u)=\alpha \}.\] Then, for a positive integer $h$, \[h*H = \{ x : (x,u) = h\alpha\}, \] so the dilation of a hyperplane results in another hyperplane which is parallel to the original one. And \[H\! -\! b = \{ x : (x,u) = \alpha - (b,u) \}\]  where $ b \in \mathbb{R}^n$\,, so the translation of a hyperplane is a hyperplane that is parallel to the original one as well. 

Now, let's calculate the distance between \[H_1 = h*H \,,\] \[ H_2 = g*H -b,\] where $h > g$, $h,g$ are positive integers, and $b \in \mathbb{R}^n $. Then $H_1 = \{ x : (x,u)=h\alpha \}$, $H_2 = \{ x : (x,u)=g\alpha - (b,u) \}$, so $H_1$ is parallel to $H_2$\,. Take any point $x_1 \in H_1$\,. Then $x_1 + tu,\, t \in \mathbb{R}$ is a ray perpendicular to both $H_1$ and $H_2$\,. Let's say $x_1 + tu \in H_2$ when $t= t_2$\,. Then \[ t_2 = \frac{(g-h)\alpha - (b,u)}{|u|^2}\,,\]
\[ d(H_1, H_2) = |t_2 u| = \frac{|(g-h)\alpha -(b,u)|}{|u|} \,. \]

Without loss of generality, we may assume $0 \in A$ because, if not, take any $a \in A$ which is also a vertex of $\Delta$. Then take $\bar{\Delta} = \Delta - a$ so that $0 \in \bar{\Delta}$. Then $\bar{\Delta} = \text{conv}(A-a)$ and $h*\bar{\Delta} = h*\Delta -ha = h* \text{conv}(A-a)$\,. And, for any positive integer $h$, if $x \in  (h*\Delta) \cap \mathbb{Z}^n$ with $d(x, \partial(h*\Delta)) > \rho$\,, then $x-ha \in h*\bar{\Delta}$, and $d(x-ha, \partial(h*\bar{\Delta})) > \rho$ since a translation doesn't change the distance. Thus $x-ha \in h(A-a)= hA-ha$. So $x \in hA$, proving our claim.

Let $h \ge C+ mC +1$. Recall \[\Delta(h,C)= C\sum a_i +(h-C-mC)*\Delta.\] Let $\Delta = G_1^+ \cap \ldots \cap G_l^+$ where $G_i$'s are hyperplanes $\{x : (x,u_i)=\alpha_i \}$ with $G_i \cap \Delta \ne \emptyset$. Then $h*\Delta = H_{1}^+ \cap \ldots \cap H_{l}^+$ and $\Delta(h,C) = H_{1}^{'+} \cap \ldots \cap H_{l}^{'+}$ where $H_i = h*G_i$\,, $H_{i}^{'} = (h-C-mC)*G_i +C\sum a_i$\,. And for all $h \ge C+mC+1$, \[d(H_i,H_{i}^{'}) = \frac{|(-C-mC)\alpha_i + (C\sum a_i\,,u_i)|}{|u_i|} \]
for $i=1, \ldots , l$, using the result above on the distance between hyperplanes. Thus, for all $i=1, \ldots ,l$, the distance $d(H_i, H_{i}^{'})$ remains same for all $h \ge C+mC+1$. 

Thus, fix any $h \ge C+mC+1$. Define \[ \rho = \max\, \{\, \delta\bigl((C+mC)*\Delta\bigr), \,d(H_i, H_{i}^{'}), i=1, \ldots , l \,\}\] where $\delta(S)$ represents the diameter of the set $S$. Then $\rho$ is independent of $h$. Let $z \in h*\Delta$ be an integral point with $d(z, \partial(h*\Delta)) > \rho$. Note that if $h \le C+mC$, then by the definition of $\rho$, such $z$ does not exist.

Let $F_i = H_i \cap (h*\Delta) \ne \emptyset$ be a face of $h*\Delta$ and $F_{i}^{'} = H_{i}^{'} \cap \Delta(h,C) \ne \emptyset$ be a face of $\Delta(h,C)$.
Assume $z \in H_1^{'-} \setminus H_{1}^{'}$. Then $d(z, H_1) < d(H_{1}^{'}, H_1) \le \rho$, but $d(z,F_1) > \rho$. Thus the perpendicular ray to $H_1$ from $z$ does not intersect $F_1$\,. It is a well known fact that every compact convex body in $\mathbb{R}^n$ with nonempty interior is homeomorphic to the closed $n$-ball, and its boundary is homeomorphic to the $(n-1)$-sphere. So \,$\partial(h*\Delta)$ is homeomorphic to the $(n-1)$-sphere. Thus, the perpendicular ray to $H_1$ from $z$ above intersects $\partial(h*\Delta)$ somewhere, say, at $z_2$ which is a point of a face $F_2$, $F_2 \ne F_1$\,. Then $z_2 \in F_2 \subseteq h*\Delta$, so $z_2 \in H_1^+$. Then $d(z, F_2) \le d(z,z_2) \le d(z,H_1) < \rho$, a contradiction. Therefore, $z \in H_1^{'+}$. Similarly, $z$ belongs to other $H_i^{'+}$ as well. Thus, $z \in H_1^{'+} \cap \ldots \cap H_l^{'+} =\Delta(h,C)$. Then, by Lemma~\ref{L:Khosec}, $z \in hA$. 
\end{proof}

By Theorem~\ref{T:Dist}, the sumset $hA$ in $\mathbb{R}^n$ takes over the central region of dilated polytopes. Han~\cite{Han} showed that, for $A \subseteq \mathbb{R}^2$ satisfying some conditions, the cardinality of $hA$ in boundary region of dilated polytopes is a linear function of $h$ when $h$ is sufficiently large. For the problems counting lattice points in "thin" annuli, Wigman~\cite{Wig} studied the statistical behavior of the counting function. It will be interesting if we can tell something more about the density or distribution of sumsets in the boundary region.

\end{document}